\documentclass[12pt]{amsart}
\usepackage{amssymb}
\usepackage{enumerate}

\newcommand{\N}{{\mathbb N}}
\newcommand{\C}{{\mathbb C}}

\newcommand{\R}{{\mathbb R}}

\newcommand{\wand}{wandering domain}

\newcommand{\tef}{transcendental entire function}

\newcommand{\nhd}{neighbourhood}
\newcommand{\sconn}{simply connected}
\newcommand{\mconn}{multiply connected}

\newcommand{\sw}{spider's web}

\newcommand{\spl}{strongly polynomial-like}

\theoremstyle{plain}
\newtheorem{theorem}{Theorem}[section]
\newtheorem{corollary}[theorem]{Corollary}
\newtheorem*{theorem*}{Theorem}
\newtheorem*{proposition*}{Proposition}

\newtheorem{lemma}[theorem]{Lemma}
\theoremstyle{definition}
\newtheorem{definition}[theorem]{Definition}
\theoremstyle{remark}
\newtheorem*{remark*}{Remark}
\newtheorem*{remarks*}{Remarks}

\theoremstyle{problem}

\theoremstyle{example}
\newtheorem{example}[theorem]{Example}
\newtheorem*{example*}{Example}

\newenvironment{myindentpar}[1]%
{\begin{list}{}%
         {\setlength{\leftmargin}{#1}}%
         \item[]%
}
{\end{list}}

\begin{document}
%Topmatter

%Two authors

\title[Connectedness properties of the set $ K(f) $ for $ f $ entire]{Connectedness properties of the set where \\ the iterates of an entire function are bounded}

\author{John Osborne}
\address{Department of Mathematics and Statistics \\
   The Open University \\
   Walton Hall\\
   Milton Keynes MK7 6AA\\
   UK}
\email{j.osborne@open.ac.uk}

%\thanks{The second author <... thanks>}

%End Two authors

%\keywords{}
%\subjclass{30D05, 37F10}
%\date{1999}

%End topmatter

\begin{abstract}
We investigate some connectedness properties of the set of points $ K(f) $ where the iterates of an entire function $ f $ are bounded.  In particular, we describe a class of \tef s for which an analogue of the Branner-Hubbard conjecture holds and show that, for such functions, if $ K(f) $ is disconnected then it has uncountably many components. We give examples to show that $ K(f) $ can be totally disconnected, and we use quasiconformal surgery to construct a function for which $ K(f) $ has a component with empty interior that is not a singleton.
\end{abstract}

\maketitle

\section{Introduction} 
\label{intro}
\setcounter{equation}{0}

Denote the $ n $th iterate of an entire function $ f $ by $ f^n $, for $ n \in \N. $  For any $ z \in \C $, we call the sequence $ (f^n(z))_{n \in \N} $ the \textit{orbit} of $ z $ under $ f $.  This paper concerns the set $ K(f) $ of points whose orbits are bounded under iteration,
\[ K(f) = \lbrace z \in \C : (f^n(z))_{n \in \N} \textrm{ is bounded} \rbrace. \]
This set has been much studied where $ f $ is a non-linear polynomial but has received less attention where $ f $ is transcendental entire. 

We assume that the reader is familiar with the main ideas of one-dimensional complex dynamics, for which we refer to \cite{aB, wB93, CG, Mil}. For convenience, we give a brief summary of relevant background and terminology at the end of this section, including definitions of the Fatou set $ F(f) $, the Julia set $ J(f) $ and the escaping set $ I(f). $   

If $ f $ is a non-linear polynomial, then $ K(f) $ is a compact set called the \textit{filled Julia set} of $ f $, and we have $ J(f) = \partial K(f) $ and $ K(f) = \C \setminus I(f) $. If $ f $ is a \tef, then it remains true that $ J(f) = \partial K(f) $ (since $ K(f) $ is completely invariant and any Fatou component that meets $ K(f) $ lies in $ K(f) $), but $ K(f) $ is not closed or bounded and is not the complement of $ I(f). $  Indeed, there are always points in $ J(f) $ that are in neither $ I(f) $ nor $ K(f) $ \cite[Lemma 1]{BD2}, and there may also be points in $ F(f) $ with the same property \cite[Example 1]{EL}.

Bergweiler \cite[Theorem 2]{wB10} has recently shown that there exist \tef s for which the Hausdorff dimension of $ K(f) $ is arbitrarily close to ~$ 0 $.  This is perhaps surprising given the result of Bara\'{n}ski, Karpi\'{n}ska and Zdunik \cite{BKZ} that the Hausdorff dimension of $ K(f) \cap J(f) $ is strictly greater than ~$ 1 $ when $ f $ is in the Eremenko-Lyubich class ~$ \mathcal{B} $ (so that the set of all critical values and finite asymptotic values of $ f $ is bounded).

These results raise questions about the topological nature of $ K(f) $ where $ f $ is transcendental entire, and in this paper we explore some of its connectedness properties.  In particular, we give some results on the number of components of $ K(f) $, and we exhibit a class of \tef s for which $ K(f) $ is totally disconnected if and only if each component of $ K(f) $ containing a critical point is aperiodic, that is, not periodic (so that an analogue of the Branner-Hubbard conjecture holds).

It is well known that, if $ f $ is a non-linear polynomial and $ K(f) $ contains all of the finite critical points of $ f $, then both $ J(f) $ and $ K(f) $ are connected, whilst if at least one finite critical point belongs to $ \C \setminus K(f) $ then each of $ J(f) $ and $ K(f) $ has uncountably many components; see, for example, Milnor \cite[Theorem 9.5]{Mil}.  

For a general \tef, Baker and Dom\'{i}nguez have shown that $ J(f) $ is either connected or has uncountably many components \cite[Theorem ~B]{BD1}, but no corresponding result is known for $ K(f) $.  However, a result of Rippon and Stallard \cite[Theorem 5.2]{RS11} easily gives the following.

\begin{theorem}
\label{conn}
Let $ f $ be a \tef.  Then $ K(f) $ is either connected or has infinitely many components. 
\end{theorem} 

A simple example of a function for which $ K(f) $ is connected is the exponential function
\[ f(z) = \lambda e^z, \textrm{ where } 0 < \lambda < 1/e. \]
Recall that, for this function, $ F(f) $ consists of the immediate basin of an attracting fixed point, so that $ F(f) \subset K(f) $.  Since $ F(f) $ is connected and $ \overline{F(f)} = \C $, it follows that $ K(f) $ is also connected.

At the other extreme, we give several examples in this paper of functions for which $ K(f) $ is totally disconnected, including the function
\[ f(z) = z + 1 + e^{-z}, \]
first studied by Fatou (see Example \ref{fatou}).

We now give a new result on the components of $ K(f) \cap J(f) $ for a general \tef, and a stronger result than Theorem \ref{conn} on the components of $ K(f) $ for a particular class of functions.

\begin{definition}
\label{spl}
We say that a \tef\ $ f $ is \textit{\spl} if there exist sequences $ (V_n), (W_n) $ of bounded, \sconn\ domains with smooth boundaries such that $ V_n \subset V_{n+1} $ and $ W_n \subset W_{n+1} $ for $ n \in \N $, $ \bigcup_{n \in \N} V_n = \bigcup_{n \in \N} W_n = \C $ and each of the triples $ (f; V_n, W_n) $ is a polynomial-like mapping in the sense of Douady and Hubbard \cite{DH}. 
\end{definition}

We prove the following.

\begin{theorem}
\label{uncount}
Let $ f $ be a \tef.
\begin{enumerate}[(a)]
\item Either $ K(f) \cap J(f) $ is connected, or else every \nhd\ of a point in $ J(f) $ meets uncountably many components of $ K(f) \cap J(f) $.
\item If $ f $ is \spl\ then either $ K(f) $ is connected, or else every \nhd\ of a point in $ J(f) $ meets uncountably many components of~$ K(f) $.
\end{enumerate}
\end{theorem}

\begin{remarks*}
\begin{enumerate}[1.]
\item We note that $ K(f) \cap J(f) $ can be connected, for example when $ f(z) = \sin z. $  For in proving the connectedness of $ J(f) $ in \cite[Theorem 4.1]{PD1}, Dom\'{i}nguez also showed that the union $ E $ of the boundaries of all Fatou components is connected.  Since, for this function, all Fatou components are bounded and $ F(f) \subset K(f) $, it follows that $ E \subset K(f) \cap~J(f) \subset J(f) $ and hence that $ K(f) \cap J(f) $ is connected.  A similar argument shows that $ K(f) $ is connected.
\item We know of no example of a \spl\ function $ f $ for which $ K(f) $ is connected.
\end{enumerate}
\end{remarks*}

Another well known result from polynomial dynamics says that, if $ f $ is a non-linear polynomial, then $ K(f) $ is totally disconnected if all of the critical points of $ f $ lie outside $ K(f) $; see for example \cite[p. 67]{CG}.  More generally, it has recently been proved \cite{QY} that if $ f $ is such a polynomial then a component of $ K(f) $ is a singleton if and only if its orbit includes no periodic component containing a critical point.  In particular, this proved the Branner-Hubbard conjecture, that $ K(f) $ is totally disconnected if and only if each component of $ K(f) $ containing a critical point is aperiodic (for another proof of the Branner-Hubbard conjecture see \cite{KvS}).  

It is natural to ask whether an analogous result holds for \tef s.  In the following theorem we show that this is the case if $ f $ is \spl.   

\begin{theorem}
\label{branhubbgen}
Let $ f $ be a \spl\ \tef\ and let $ K $ be a component of $ K(f). $
\begin{enumerate}[(a)]
\item The component $ K $ is a singleton if and only if the orbit of $ K $ includes no periodic component of $ K(f) $ containing a critical point.  In particular, if $ K $ is a wandering component of $ K(f) $, then $ K $ is a singleton.
\item The interior of $ K $ is either empty or consists of bounded, non-wandering Fatou components.  If these Fatou components are not Siegel discs, then they are Jordan domains.
\end{enumerate}
\end{theorem}

\begin{corollary}
\label{totdisc}
Let $ f $ be a \spl\ \tef. 
\begin{enumerate}[(a)]
\item All except at most countably many components of $ K(f) $ are singletons.
\item $ K(f) $ is totally disconnected if and only if each component of $ K(f) $ containing a critical point is aperiodic. 
\end{enumerate}
\end{corollary}

The following alternative characterization of \spl\ functions is useful for checking that functions are \spl, and may be of independent interest. Here and elsewhere in the paper we say that a set $ S \subset \C $ \textit{surrounds} a set or a point if that set or point lies in a bounded complementary component of $ S $.  

\begin{theorem}
\label{altdef}
A \tef\ $ f $ is \textit{\spl} if and only if there exists a sequence of bounded, \sconn\ domains $ (D_n)_{n \in \N} $ such that
\begin{itemize}
\item $ \overline{D}_n \subset D_{n+1} $, for $ n \in \N $,
\item $ \bigcup_{n \in \N} D_n = \C $, and
\item $ f(\partial D_n) $ surrounds $ \overline{D}_n $, for $ n \in \N. $
\end{itemize}
\end{theorem}

Our final result shows that there are large classes of \tef s which have the property of being \spl.  The terminology used in this theorem is explained in Section \ref{splfunc}.

\begin{theorem}
\label{types}
A \tef\ $ f $ is \textit{\spl} if there exists an unbounded sequence $ (r_n) $ of positive real numbers such that
\begin{equation*}
m(r_n,f) := \min \lbrace \vert f(z) \vert : \vert z \vert = r_n \rbrace > r_n, \quad \textrm{for } n \in \N. 
\end{equation*}
In particular, this is the case if one of the following conditions holds:
\begin{enumerate}[(a)]
\item $ f $ has a multiply connected Fatou component;
\item $ f $ has growth not exceeding order $ \tfrac{1}{2} $, minimal type;
\item $ f $ has finite order and Fabry gaps;
\item $ f $ has a sufficiently strong version of the pits effect.
\end{enumerate}
\end{theorem}

\begin{remark*}
In the following notes, we clarify the relationship between the results  in this paper for \spl\ functions, and earlier results for \tef s with the property that a certain subset $ A_R(f) $ of the escaping set has a geometric form known as a \sw\ (we refer to \cite{RS10a} for the terminology used here).
\begin{itemize}
\item It follows from Theorem \ref{altdef} and \cite[Lemma 7.2]{RS10a} that if $ A_R(f) $ is a \sw\ then  $ f $ is \spl.  However, the converse is not true - see \cite[Theorem~1.2]{RS12}.  
\item Theorem~\ref{types} is similar to \cite[Theorem 1.9]{RS10a}, which gave various classes of functions for which $ A_R(f) $ is a \sw.  However, in Theorem~\ref{types} we do not need the additional regular growth condition that was required for several of the function classes in \cite[Theorem 1.9]{RS10a}.
\item Theorem \ref{branhubbgen} is a generalisation to \spl\ functions of results previously proved for functions with an $ A_R(f) $ \sw\ in \cite[Theorem~1.5]{O10}.   
\end{itemize}
\end{remark*}

The organisation of this paper is as follows.  In Section \ref{thms1}, we recall the definition of a polynomial-like mapping and prove the analogue of the Branner-Hubbard conjecture for \spl\ functions (Theorem \ref{branhubbgen} and Corollary~\ref{totdisc}).  Section \ref{thms2} contains the proofs of our results on the number of components of $ K(f) $ (Theorems \ref{conn} and \ref{uncount}).  In Section \ref{splfunc}, we prove Theorems \ref{altdef} and \ref{types} on \spl\ functions.  In Section \ref{examples}, we give several examples of \tef s for which $ K(f) $ is totally disconnected.  Finally, in Section \ref{further}, we use quasiconformal surgery to construct a \tef\ for which $ K(f) $ has a component with empty interior which is not a singleton.

\textbf{Background and terminology} 

We summarise here some ideas and terminology from one-dimensional complex dynamics that are used throughout this paper.  In what follows, $ f $ is an entire function.  
 
The \textit{Fatou set} $ F(f) $ is the set of points $ z \in \C $ such that the family of functions $ \lbrace f^n : n \in \N \rbrace $ is normal in some \nhd\ of $ z $, and the \textit{Julia set} $ J(f) $ is the complement of $ F(f) $.  The \textit{escaping set} $ I(f) $ is the set of points whose orbits tend to infinity, 
\[ I(f) = \lbrace z \in \C : f^n(z) \to \infty \text{ as } n \to \infty \rbrace. \]
If we say that the set $ S $ is \textit{completely invariant} under a function $ f $, we mean that $ z \in S $ if and only if $ f(z) \in S $.  Each of the sets $ J(f), F(f), I(f) $ and $ K(f) $ is completely invariant.
  
A component of the Fatou set $ F(f) $ is often referred to as a \textit{Fatou component}.  If $ U = U_0 $ is a Fatou component, then for each $ n \in \N $, $ f^n(U) \subset U_n $ for some Fatou component $ U_n $.  If $ U = U_n $ for some $ n \in \N $, we say that $ U $ is \textit{periodic}; otherwise, we say that it is \textit{aperiodic}.  If $ U $ is not eventually periodic, i.e. if $ U_m \neq U_n $ for all $ n > m \geq 0 $, then $ U $ is called \textit{wandering}. Wandering Fatou components can occur for \tef s but not for polynomials \cite{iB84, SU}.   There are four possible types of periodic Fatou components for a \tef, namely immediate attracting basins, immediate parabolic basins, Siegel discs and Baker domains.  We refer to \cite{wB93} for the definitions and properties of such components. 

If $ K $ is a component of $ K(f) $, we call the sequence of components $ K_n $ such that $ f^n(K) \subset K_n $  the \textit{orbit} of $ K $. Periodic, aperiodic and wandering components of $ K(f) $ are defined as for components of $ F(f). $  Periodic components of $ K(f) $ always exist and wandering components may exist, both for polynomials (since at most countably many components of $ J(f) $ are eventually periodic; see, for example, \cite{Mc}) and for \tef s (see, for example, \cite[Theorem 1.2]{O10}).

The dynamical behaviour of an entire function $ f $ is much affected by its critical values and finite asymptotic values. If $ f'(z) = 0 $ we say that $ z $ is a \textit{critical point} and $ f(z) $ is a \textit{critical value} of $ f $.  A \textit{finite asymptotic value} of $ f $ is a point $ a \in \C $ such that there is a curve $ \gamma:[0,\infty) \to \C $ with $ \gamma(t) \to \infty $ and $ f(\gamma(t)) \to a $ as $ t \to \infty. $  Finite asymptotic values can occur for \tef s but not for polynomials.

\section{Proofs of theorem \ref{branhubbgen} and corollary \ref{totdisc}} 
\label{thms1}
\setcounter{equation}{0}

In this section, we prove the analogue of the Branner-Hubbard conjecture for \spl\ functions, and some related results (Theorem \ref{branhubbgen} and Corollary \ref{totdisc}).  

First, we recall Douady and Hubbard's definition of a polynomial-like mapping and its filled Julia set (see Chapter VI of \cite{CG}, and \cite{DH}).

\begin{definition}
Let $ V $ and $ W $ be bounded, \sconn\ domains with smooth boundaries such that $ \overline{V} \subset W $.  Let $ f $ be a proper analytic mapping of $ V $ onto $ W $ with $ d $-fold covering, where $ d \geq 2. $  Then the triple $ (f; V, W) $ is termed a \textit{polynomial-like mapping} of degree $ d $.  The \textit{filled Julia set} $ K(f; V, W) $ of the polynomial-like mapping $ (f; V, W) $ is defined to be the set of all points whose orbits lie entirely in $ V $, i.e.  
\[  K(f; V, W) = \bigcap_{k \geq 0} f^{-k} (V).\]
\end{definition}

The proof of Theorem \ref{branhubbgen} relies on Douady and Hubbard's Straightening Theorem for polynomial-like mappings, which is as follows.

\begin{theorem}\cite[Theorem 1]{DH}
\label{straight}
If $ (f; V, W) $ is a polynomial-like mapping of degree $ d \geq 2 $, then there exists a quasiconformal mapping $ \phi : \C \to \C $ and a polynomial $ g $ of degree $ d $ such that $ \phi \circ f = g \circ \phi $ on $ \overline{V} $.  Moreover
\[ \phi ( K(f; V, W)) = K(g),  \]
where $ K(g) $ is the filled Julia set of the polynomial $ g. $
\end{theorem} 

We also need the following recent results from polynomial dynamics.  

\begin{theorem} \cite{QY}
\label{qy}
For a non-linear polynomial $ g $, a component of $ K(g) $ is a singleton if and only if its orbit includes no periodic component of $ K(g) $ containing a critical point. 
\end{theorem}

\begin{theorem} \cite{RY1, RY2}
\label{ry}
If $ g $ is a non-linear polynomial, then any bounded component of $ F(g) $ which is not a Siegel disc is a Jordan domain. 
\end{theorem}

Finally, we make use of the following topological result.

\begin{lemma}
\label{union}
A countable union of compact, totally disconnected subsets of $ \C $ is totally disconnected.
\end{lemma}

\begin{proof} 
This is an immediate consequence of the following results, which may be found in Hurewicz and Wallman \cite[Chapter II]{HW}: 
\begin{itemize}
\item a compact, separable metric space is totally disconnected if and only if it is $ 0 $-dimensional;
\item a separable metric space which is the countable union of $ 0 $-dimensional closed subsets of itself is $ 0 $-dimensional; 
\item every $ 0 $-dimensional, separable metric space is totally disconnected.
\end{itemize}
Here, a non-empty space is \textit{0-dimensional} if each of its points has arbitrarily small \nhd s with empty boundaries.
\end{proof}

We are now in a position to give the proof of Theorem \ref{branhubbgen} and Corollary \ref{totdisc}.

\begin{proof}[Proof of Theorem \ref{branhubbgen}]
Since $ f $ is \spl, it follows from Definition \ref{spl} that there exist sequences $ (V_n), (W_n) $ of bounded, \sconn\ domains with smooth boundaries such that $ V_n \subset V_{n+1} $ and $ W_n \subset W_{n+1} $ for $ n \in \N $, $ \bigcup_{n \in \N} V_n = \bigcup_{n \in \N} W_n = \C $ and each of the triples $ (f; V_n, W_n) $ is a polynomial-like mapping.

Let $ K(f; V_n, W_n) $ denote the filled Julia set of the polynomial-like mapping $ (f; V_n, W_n) $. Then clearly we have
\[ K(f; V_n, W_n) \subset K(f; V_{n+1}, W_{n+1}), \quad \textrm{ for } n \in \N, \]
and
\begin{equation}
\label{ksum}
K(f) = \bigcup_{n \in \N} K(f; V_n, W_n).  
\end{equation} 

Now let $ K $ be a component of $ K(f) $ whose orbit includes no periodic component of $ K(f) $ containing a critical point. We show that $ K $ must be a singleton.   

For each $ n \in \N $, define
\[ K_n = K \cap K(f; V_n, W_n).  \]
Then $ K = \bigcup_{n \in \N} K_n $, and since any component of $ K(f; V_n, W_n) $ must lie in a single component of $ K(f) $ it follows that, where $ K_n \neq \emptyset, $ each component of $ K_n $ must be a component of $ K(f; V_n, W_n) $. In particular, each component of $ K_n $ must be compact.

Moreover, no component of $ K_n $ can have an orbit which includes a periodic component of $ K(f; V_n, W_n) $ containing a critical point.  For any such periodic component of $ K(f; V_n, W_n) $ would lie in a periodic component of $ K(f), $ and since $ K_n \subset K $, the orbit of $ K $ would then include a periodic component of $ K(f) $ containing a critical point, contrary to our assumption.

Now it follows from Theorem \ref{straight} that, for each $ n \in \N, $ there exists a quasiconformal mapping $ \phi_n : \C \to \C $ and a polynomial $ g_n $ of the same degree as $ (f; V_n, W_n) $ such that $ \phi_n \circ f = g_n \circ \phi_n $ on $ \overline{V}_n, $  and
\begin{equation}
\label{qch}
\phi_n(K(f; V_n, W_n)) = K(g_n), 
\end{equation}
where $ K(g_n) $ is the filled Julia set of the polynomial $ g_n $.  

Thus it follows from (\ref{qch}) and Theorem \ref{qy}, and the fact that critical points are preserved by the quasiconformal mapping, that every component of $ K_n $ is a singleton, i.e. $ K_n $ is totally disconnected, for each $ n \in \N $.  Lemma \ref{union} now gives that $ K $ is totally disconnected, and since $ K $ is connected it must be a singleton.

For the converse, suppose now that a component $ K $ of $ K(f) $ is a singleton.  Then it follows from (\ref{ksum}) that there exists $ N \in \N $ such that $ K $ is a singleton component of $ K(f; V_n, W_n) $ for all $ n \geq N $.  Thus, by (\ref{qch}) and Theorem \ref{qy}, for each $ n \geq N $ the orbit of $ K $ can include no periodic component of $ K(f; V_n, W_n) $ containing a critical point.  The desired converse now follows from (\ref{ksum}).    

Finally, since by definition the orbit of a wandering component of $ K(f) $ contains no periodic component, it follows that every wandering component of $ K(f) $ is a singleton.  This completes the proof of part $ (a) $. 

To prove part $ (b) $ note first that, for any \tef\ $ f $, since $ J(f) = \partial K(f) $ it is immediate that for any component $ K $ of $ K(f) $ we have $ \partial K \subset J(f) $ and $ \textrm{int}(K) \subset F(f). $

Now let $ f $ be \spl, and let $ K $ be a component of $ K(f) $ with non-empty interior.  As in the proof of part $ (a) $, we write $ K = \bigcup_{n \in \N} K_n $ where
\[ K_n = K \cap K(f; V_n, W_n),  \]
so $ K_n $ has non-empty interior for sufficiently large $ n $.  Then, since every component of $ K_n $ is a component of $ K(f; V_n, W_n) $, it follows from (\ref{qch}) that the interior of a component of $ K_n $ is quasiconformally homeomorphic to the interior of a component of the filled Julia set $ K(g_n) $ of the polynomial $ g_n, $ which consists of bounded Fatou components that are non-wandering by Sullivan's theorem~\cite{SU}.  Evidently, therefore, if a Fatou component $ U $ of $ f $ meets $ K_n $, we have $ \overline{U} \subset K_n $, and it follows that all Fatou components in $ K(f) $ are bounded and non-wandering. Since Siegel discs and Jordan curves are preserved by the quasiconformal mapping, the remainder of part (b) now follows from Theorem~\ref{ry}. 
\end{proof}

\begin{proof}[Proof of Corollary \ref{totdisc}]
Since $ f $ is \spl, it follows from Theorem \ref{branhubbgen}(a) that a component $ K $ of $ K(f) $ is a singleton unless the orbit of $ K $ includes a periodic component of $ K(f) $ containing a critical point.  Part (a) now follows because $ f $ can have at most countably many critical points.

If $ K(f) $ is totally disconnected then all of its components are singletons, so part~(b) follows immediately from Theorem \ref{branhubbgen}(a).  
\end{proof}

\begin{remark*}
Zheng \cite[Theorem 2]{Z0}, \cite[Theorem 4]{Z1} has shown that, if $ f $  is a \tef\ for which there exists an unbounded sequence $ (r_n) $ of positive real numbers such that
\begin{equation*}
m(r_n,f) > r_n, \quad \textrm{for } n \in \N, 
\end{equation*}
and if $ U $ is a component of $ F(f) $, then
\begin{enumerate}[(i)]
\item if $ U $ contains a point $ z_0 $ such that $ \lbrace f^n(z_0) : n \in \N \rbrace $ is bounded, then $ U $ is  bounded, and
\item if $ U $ is wandering, then there exists a subsequence of $ f^n $ on $ U $ tending to~$ \infty. $
\end{enumerate}  
It follows that, for such functions, the interior of $ K(f) $ consists of bounded, non-wandering Fatou components.  As these functions are \spl\ by Theorem \ref{types}, the first part of Theorem \ref{branhubbgen}(b) is a generalisation of Zheng's results. 
\end{remark*}

\section{Proofs of theorems \ref{conn} and \ref{uncount}} 
\label{thms2}
\setcounter{equation}{0}

In this section we prove Theorems \ref{conn} and \ref{uncount}, which concern the number of components of $ K(f) $ and $ K(f) \cap J(f) $ when $ f $ is a \tef.
 
Theorem \ref{conn} is a consequence of the following result due to Rippon and Stallard.  Here $ E(f) $ is the \textit{exceptional set} of $ f $, i.e. the set of points with a finite backwards orbit under $ f $ (which for a \tef\ contains at most one point).  

\begin{theorem}\cite[Theorem 5.2]{RS11} 
\label{rs}
Let $ f $ be a \tef.  Suppose that the set $ S $ is completely invariant under $ f $, and that $ J(f) = \overline{S \cap J(f)}. $  Then exactly one of the following holds:
\begin{enumerate}[(1)]
\item $ S $ is connected;
\item $ S $ has exactly two components, one of which is a singleton $ \lbrace \alpha \rbrace $, where $ \alpha $ is a fixed point of $ f $ and $ \alpha \in E(f) \cap F(f); $
\item $ S $ has infinitely many components.
\end{enumerate}
\end{theorem}

\begin{proof} [Proof of Theorem \ref{conn}]
Since $ K(f) $ is completely invariant and dense in $ J(f) $, it is evident that the conditions of Theorem \ref{rs} hold with $ S = K(f) $.  Case (2) cannot occur since if $ z \in F(f) $ has bounded orbit, then so does a \nhd\ of $ z $ in $ F(f). $ 
\end{proof}

Theorem \ref{uncount} gives a new result on components of $ K(f) \cap J(f) $ for a general \tef, and also shows that we can improve on Theorem~\ref{conn} for \spl\ functions.  Our proof of this result uses the well-known \textit{blowing up property} of $ J(f) $:

\begin{myindentpar}{1cm}
if $ f $ is an entire function, $ K $ is a compact set, $ K \subset \C \setminus E(f) $ and $ G $ is an open \nhd\ of $ z \in J(f) $, then there exists $ N \in \N $ such that $ f^n(G) \supset K $, for all $ n \geq N. $
\end{myindentpar}

We also need the following lemmas.

\begin{lemma} \cite[Lemma 3.1]{R11}
\label{rempe}
Let $ C \subset \C $.  Then $ C $ is disconnected if and only if there is a closed connected set $ A \subset \C $ such that $ C \cap A = \emptyset $ and at least two different connected components of $ \C \setminus A $ intersect $ C. $
\end{lemma} 

\begin{lemma} \cite[Lemma 1]{RS09}
\label{cover}
Let $ E_n, n \geq 0, $ be a sequence of compact sets in $ \C $, and $ f : \C \to \widehat{\C} $ be a continuous function such that
\[ f(E_n) \supset E_{n+1}, \text{   for } n \geq 0. \]
Then there exists $ \zeta $ such that $ f^n(\zeta) \in E_n $, for $ n \geq 0 $.

If $ f $ is also meromorphic and $ E_n \cap J(f) \neq \emptyset $ for $ n \geq 0 $, then there exists $ \zeta \in J(f) $ such that $ f^n(\zeta) \in E_n, $ for $ n \geq 0. $ 
\end{lemma} 

\begin{proof}[Proof of Theorem \ref{uncount}]
We first prove part (a).  If $ K(f) \cap J(f) $ is disconnected, then it follows from Lemma \ref{rempe} that there exists a continuum $ \Gamma \subset (K(f) \cap J(f))^c $ with two complementary components, $ G_1 $ and $ G_2 $ say, each of which contains points in $ K(f) \cap J(f) $.  

Suppose, then, that $ z_i \in G_i \cap K(f) \cap J(f) $ for $ i = 1,2 $, and let $ H_i $ be a bounded open \nhd\ of $ z_i $ compactly contained in $ G_i $.  Since $ J(f) $ is perfect we may without loss of generality assume that neither $ \overline{H}_1 $ nor $ \overline{H}_2 $ meets $ E(f) $.  

Now let $ z $ be an arbitrary point in $ J(f) $, and let $ V $ be a bounded open \nhd\ of $ z $.  Then, by the blowing up property of $ J(f) $, there exists $ K \in \N $ such that
\begin{equation}
\label{blowup1}
f^k(V) \supset \overline{H}_1 \cup \overline{H}_2
\end{equation}
for all $ k \geq K. $  Furthermore, there exists $ M \geq K $ such that
\begin{equation}
\label{blowup2}
f^m(H_1) \supset \overline{H}_1 \cup \overline{H}_2 \quad \textrm{ and } \quad f^m(H_2) \supset \overline{H}_1 \cup \overline{H}_2,
\end{equation}
for all $ m \geq M. $

Now let $ s = s_1s_2s_3 \ldots $ be an infinite sequence of $ 1 $s and $ 2 $s.  We show that each such sequence $ s $ can be associated with the orbit of a point in $ \overline{V} \cap K(f) \cap J(f) $, as follows.

Put $ S_0 = \overline{V} $ and, for $ n \in \N $, put $ S_n = \overline{H}_i $ if $ s_n = i $.  It follows from (\ref{blowup1}), (\ref{blowup2}) and Lemma \ref{cover} that there exists a point $ \zeta_s \in J(f) $ such that $ f^{Mn}(\zeta_s) \in S_n $ for $ n \geq 0. $  In particular, $ \zeta_s \in \overline{V}. $ Furthermore, for all $ k \geq 0 $ we have
\[ f^k(\zeta_s) \in \bigcup_{j = 0}^{M - 1} f^j(\overline{V}) \cup f^j(\overline{H}_1 \cup \overline{H}_2), \]
so $ \zeta_s $ has bounded orbit and thus lies in $ K(f). $

Now the points in $ \overline{V} \cap K(f) \cap J(f) $ whose orbits are associated with two different infinite sequences of $ 1 $s and $ 2 $s must lie in different components of $ K(f) \cap J(f) $.  For if two such sequences first differ in the $ N $th term, then the $ MN $th iterate of one point will lie in $ G_1 $ and the other in $ G_2 $.  Thus, if the two points were in the same component $ K $ of $ K(f) \cap J(f) $, then $ f^{MN}(K) $ would meet $ \Gamma \subset (K(f)\cap J(f))^c, $ which is a contradiction.

Now there are uncountably many possible infinite sequences $ s = s_1s_2s_3 \ldots $ of $ 1 $s and $ 2 $s, so we have shown that every \nhd\ of an arbitrary point in $ J(f) $ meets uncountably many components of $ K(f) \cap J(f) $, as required.  

The proof of part (b) is similar, but we now make the additional assumption that $ f $ is \spl.  Since we are assuming that $ K(f) $ is disconnected, it follows from Lemma \ref{rempe} that there is a continuum in $ K(f)^c $ with two complementary components, each of which contains points in $ K(f) $.  As in the proof of part (a), we label the continuum $ \Gamma $ and the complementary components $ G_1 $ and~$ G_2 $.  

We show that, in fact, each of $ G_1 $ and $ G_2 $ must contain points in $ K(f) \cap J(f). $  For if not, $ G_i \subset F(f) $ for some $ i \in \lbrace 1, 2 \rbrace $.  However, since $ f $ is \spl, it follows from Theorem \ref{branhubbgen}(b) that the Fatou component $ U $ containing $ G_i $ must be bounded and non-wandering, so that $ \overline{U} \subset K(f) $.  Thus $ \overline{U} \subset G_i $, which is a contradiction.

So, as before, we may choose $ z_i \in G_i \cap K(f) \cap J(f) $ for $ i = 1,2 $, and bounded open \nhd s $ H_i $ of $ z_i $ compactly contained in $ G_i $.  The proof now proceeds exactly as for the proof of part (a), but we conclude that points in $ \overline{V} \cap K(f) $ whose orbits are associated with two different infinite sequences of $ 1 $s and $ 2 $s must lie in different components of $ K(f) $.  It then follows that every \nhd\ of an arbitrary point in $ J(f) $ meets uncountably many components of $ K(f) $.    
\end{proof}

\begin{remark*}
It follows from Theorem \ref{uncount}(b) and Corollary \ref{totdisc}(a) that, if $ f $ is \spl, then $ K(f) $ has uncountably many singleton components.
\end{remark*}

\section{\spl\ functions} 
\label{splfunc}
\setcounter{equation}{0}

In this section we prove Theorem \ref{altdef}, which gives a useful equivalent characterization of a \spl\ function, and Theorem \ref{types}, which gives several large classes of \tef s which are \spl. 

\begin{proof}[Proof of Theorem \ref{altdef}]

First, suppose that $ f $ is \spl\ and let $ (V_n), (W_n) $ be the sequences of bounded, \sconn\ domains in Definition ~\ref{spl}.  Since $ (f; V_n,W_n) $ is a polynomial-like mapping, it follows that $ \overline{V}_n \subset W_n $ and $ f(\partial V_n) = \partial W_n $, for $ n \in \N $.  Moreover, taking a subsequence of $ (V_n)_{n \in \N} $ if necessary, we can assume that $ W_n \subset V_{n+1} $ for $ n \in \N $.  Putting $ D_n = V_n $ for $ n \in \N $ then gives a sequence of domains with the properties stated in the theorem.

For the converse, let $ (D_n)_{n \in \N} $ be a sequence of bounded, \sconn\ domains with the properties stated in the theorem.  Since $ f(D_n) $ is bounded, we may assume without loss of generality that
\begin{equation}
\label{surround}
f(D_n) \subset D_{n+1}, \quad \textrm{ for } n \in \N.
\end{equation}
Now, for each $ n \in \N, $ let $ \Gamma_n $ be a smooth Jordan curve that surrounds $ \overline{D}_{n+1} $ and lies in the complementary component of $ f(\partial D_{n+1}) $ containing $ \overline{D}_{n+1} $.  Observe that it follows from the properties of the sequence $ (D_n)_{n \in \N} $ that $ f $ has no finite asymptotic values.  Furthermore, we may assume that each $ \Gamma_n $ does not meet any of the critical values of $ f $.    

Let $ W_n $ denote the bounded complementary component of $ \Gamma_n $. Then $ W_n $ contains $ D_{n+1} $ and hence $ f(D_n) $ by (\ref{surround}).  Thus there is a component $ V_n $ of $ f^{-1}(W_n) $ that contains $ D_n $.  Furthermore, $ f:V_n \to W_n $ is a proper mapping, and since $ f $ is transcendental we may assume that the degree of this mapping is at least $ 2. $ 

Now $ \overline{V}_n \subset W_n. $  For suppose not. Then since $ \partial D_{n+1} \subset W_n $ and $ D_n \subset V_n \cap D_{n+1} $ we must have $ V_n \cap \partial D_{n+1} \neq \emptyset $.  However, if $ \zeta \in V_n \cap \partial D_{n+1} $ then it follows that $ f(\zeta) \in W_n \cap f(\partial D_{n+1}) $, which contradicts the fact that $ W_n $ and $ f(\partial D_{n+1}) $ are disjoint.

Moreover, $ V_n $ is simply connected. For suppose that $ V_n $ is \mconn, and let $ \gamma $ be a Jordan curve in $ V_n $ which is not null homotopic there. Let $ G $ be the bounded complementary component of $ \gamma $, so that $ G $ contains a component of $ \partial V_n. $  Now since $ f $ is a proper mapping we have $ f(\partial V_n) = \Gamma_n = \partial W_n $, so $ f(G) \cap \Gamma_n \neq \emptyset $, which is impossible because $ f(\gamma) \subset W_n $ and $ f(G) $ is bounded. Thus $ V_n $ is indeed \sconn, and since $ \Gamma_n $ meets no critical values of $ f $, $ \partial V_n $ is a smooth Jordan curve.
 
This establishes that, for each $ n \in \N $, the triple $ (f; V_n, W_n) $ is a polynomial-like mapping.  Furthermore, it follows from the construction that the sequences $ (V_n) $ and $ (W_n) $ have the properties in Definition \ref{spl}.  This completes the proof.
\end{proof}

We now turn to Theorem \ref{types}, which gives a sufficient condition for a \tef\ to be \spl, and lists a number of classes of functions for which this condition holds.  The sufficient condition is proved in the following lemma.  

\begin{lemma}
\label{suffcon}
A \tef\ $ f $ is \textit{\spl} if there exists an unbounded sequence $ (r_n) $ of positive real numbers such that
\begin{equation*}
m(r_n,f) > r_n, \quad \textrm{for } n \in \N. 
\end{equation*}
\end{lemma}  

\begin{proof}
We may assume without loss of generality that the sequence $ (r_n) $ is strictly increasing.  Putting $ D_n = \lbrace z : \vert z \vert < r_n \rbrace, $ we then have $ \overline{D}_n \subset D_{n+1} $, for $ n \in \N $, and $ \bigcup_{n \in \N} D_n = \C $.  Moreover, since a \tef\ always has points of period $ 2 $, $ f(\partial D_n) $ must surround $ \overline{D}_n $ for sufficiently large $ n $.  The result now follows from Theorem \ref{altdef}.   
\end{proof}

To complete the proof of Theorem \ref{types}, we discuss in turn each of the four classes of functions listed in the theorem and show that they meet the condition in Lemma \ref{suffcon}.

First, we consider \tef s with a \mconn\ Fatou component (Theorem \ref{types}(a)).  We state some results on such components which are useful here and in subsequent sections of this paper.

The basic properties of \mconn\ Fatou components for a \tef\ were proved by Baker.

\begin{lemma} \cite[Theorem 3.1]{iB84}
\label{baker}
Let $ f $ be a \tef\ and let $ U $ be a \mconn\ Fatou component.  Then
\begin{itemize}
\item $ f^n(U) $ is bounded for any $ n \in \N $,
\item $ f^{n+1} (U) $ surrounds $ f^n(U) $ for large $ n $, and
\item $ \mathrm{dist }(0, f^n(U)) \to \infty $ as $ n \to \infty. $
\end{itemize}
\end{lemma}

Later results have shown that the iterates of a \mconn\ Fatou component eventually contain very large annuli.  The following special case of a result of Zheng \cite{Z2} is quoted in this form by Bergweiler, Rippon and Stallard in \cite{BRS10}.

\begin{lemma} \label{zheng}
Let $ f $ be a \tef\ with a \mconn\ Fatou component $ U $.  If $ A \subset U  $ is a domain containing a closed curve that is not null-homotopic in $ U $ then, for sufficiently large $ n \in \N, $
\[ f^n(U) \supset f^n(A) \supset \lbrace z \in \C : \alpha_n < \vert z \vert < \beta_n \rbrace, \]
where $ \beta_n/\alpha_n \to \infty $ as $ n \to \infty. $
\end{lemma}

Maintaining the notation of Lemmas \ref{baker} and \ref{zheng}, it follows that, for sufficiently large $ n $, 
\[ f^{n+1}(U) \textrm{ surrounds } f^n(U) \textrm{ which contains } \lbrace z \in \C : \alpha_n < \vert z \vert < \beta_n \rbrace.\]
Thus, for these values of $ n $, $ m(r,f) > r $ whenever $ \alpha_n < r < \beta_n $, so the condition in Lemma \ref{suffcon} is satisfied.

Next, we consider \tef s of growth not exceeding order ~$ \tfrac{1}{2} $, minimal type (Theorem \ref{types}(b)).  If
$ M(r,f) := \max{ \lbrace \vert f(z) \vert : \vert z \vert = r \rbrace }, $ the \textit{order} $ \rho(f) $ and the \textit{type} $ \tau(f) $ of an entire function $ f $ are defined by
\[ \rho(f) := \limsup_{r \to \infty} \dfrac{\log \log M(r,f)}{\log r},\]
and
\[ \tau(f) := \limsup_{r \to \infty} \dfrac{\log M(r,f)}{r^\rho}.\]
If $ \tau(f) = 0, $ $ f $ is said to be of \textit{minimal type}.

The following lemma implies Theorem \ref{types}(b) immediately.   

\begin{lemma}
\label{minunbound}
Let $ f $ be a \tef\ of growth not exceeding order $ \tfrac{1}{2}, $ minimal type, and let $ n \in \lbrace0, 1, \ldots \rbrace $.  Then
\[ \limsup_{r \to \infty} \dfrac{m(r,f)}{r^n} = \infty. \]
\end{lemma}

This well-known result is proved for the case $ n = 0 $ and $ \rho(f) < \tfrac{1}{2} $ in \cite[p. 274]{Ti}.  The proof in the case of order $ \tfrac{1}{2}, $ minimal type, is similar, and the case $ n > 0 $ follows by a standard argument; see, for example, \cite[p.193]{AH}.   

Finally, we consider \tef s of finite order and with Fabry gaps (Theorem \ref{types}(c)) or with a sufficiently strong version of the pits effect (Theorem \ref{types}(d)).

A \tef\ $ f $ has \textit{Fabry gaps} if 
\[ f(z) = \sum_{k = 0}^\infty a_kz^{n_k}, \]
where $ n_k/k \to \infty $ as $ k \to \infty. $  Loosely speaking, a function exhibits the \textit{pits effect} if it has very large modulus except in small regions (pits) around its zeros.  For further details of the pits effect, we refer to the discussion in \cite[Section 8]{RS10a}.

It is noted in \cite[Section 8]{RS10a} that, if $ f $ has finite order and Fabry gaps, or if $ f $ exhibits a sufficiently strong version of the pits effect, then for some $ p > 1 $ and all sufficiently large $ r $
\begin{equation}
\label{minmodcond}
\textrm{ there exists } r' \in (r, r^p) \textrm{ with } m(r', f) \geq M(r, f). 
\end{equation} 
It follows that, for these functions also, the condition in Lemma \ref{suffcon} is satisfied.  This completes the proof of Theorem \ref{types}.

\begin{remark*}
It is noted in \cite[Section 8]{RS10a} that (\ref{minmodcond}) also holds for certain functions of infinite order which satisfy a suitable gap series condition.  Evidently, these functions also are \spl.  
\end{remark*}

\section{Examples for which $ K(f) $ is totally disconnected} 
\label{examples}
\setcounter{equation}{0}

In this section and the next we illustrate our results with a number of examples.  The examples in this section are of \tef s for which $ K(f) $ is totally disconnected.  In Section 6, we give an example of a \tef\ for which $ K(f) $ has a component with empty interior which is not a singleton. 

\begin{example} 
\label{bakerdom}
Let $ f $ be the \tef\ constructed by Baker and Dom\'{i}nguez in \cite[Theorem G]{BD1}.  Then $ K(f) $ is totally disconnected.  
\end{example}

\begin{proof}
The function $ f $ constructed in \cite[Theorem G]{BD1} takes the form
\[ f(z) = k \prod_{n = 1}^\infty \left( 1 + \dfrac{z}{r_n} \right)^2, \quad 0 < r_1 < r_2 < \cdots, \quad k > 0, \]
where the constants $ k $ and $ r_n, n \in \N, $ are chosen so that $ f(x) > x $ for $ x \in \R $ and so that the annuli
\[ A_n = \left \lbrace z : 2r_n^2 < \vert z \vert < \left ( \dfrac{r_{n+1}}{2} \right )^{1/2} \right \rbrace \]
are disjoint, with $ f(A_n) \subset A_{n+1} $ for large $ n $ (we refer to \cite[proof of Theorem G]{BD1} for details of the construction).  

As noted in \cite{BD1}, $ f $ has order zero.  Thus $ f $ is \spl, by Theorem \ref{types}(b).  Furthermore, the construction ensures that $ f(x) > x $ for $ x \in \R $, so it is easy to see that $ \R \subset I(f) $.  Since all critical points of $ f $ lie on the negative real axis, it follows that none are in $ K(f) $ and hence that $ K(f) $ is totally disconnected by Corollary \ref{totdisc}(b).
\end{proof}

The function in Example \ref{bakerdom} has \mconn\ Fatou components.  This fact gives an alternative method of showing that $ K(f) $ is totally disconnected by using results due to Kisaka \cite{K3} (see \cite[Section 5]{O10} for a discussion of these results).  Recall that a \textit{buried point} is a point in the Julia set that does not lie on the boundary of a Fatou component, and that a \textit{buried component} of the Julia set is a component consisting entirely of buried points. In \cite[Corollary D]{K3} Kisaka proved that, if a \tef\ has a \mconn\ Fatou component and each critical point has an unbounded forward orbit, then every component of the Julia set with bounded orbit must be a buried singleton component.  In \cite[Example E]{K3}, he showed that this result applies to the function $ f $ in Example \ref{bakerdom}.  Since, for this function, no component of $ J(f) $ with bounded orbit meets the boundary of a Fatou component, it follows that $ K(f) \subset J(f) $ and hence that $ K(f) $ is totally disconnected.       

In our next example, $ K(f) $ is again totally disconnected, but this time $ f $ has no \mconn\ Fatou components.  

\begin{example} 
\label{infprod}
Define $ f $ by
\[ f(z) = \prod_{n = 1}^\infty \left( 1 + \dfrac{z}{2^n} \right)^2. \]
Then $ K(f) $ is totally disconnected.   Moreover, $ f $ has no \mconn\ Fatou components. 
\end{example}

\begin{proof}
Since the zeros of $ f $ are at $ z = -2^n, n \in \N, $ $ f $ has order zero and so is \spl\ by Theorem \ref{types}(b).  Furthermore, for $ x \in \R, $
\[ f(x) \geq \left( 1 + \dfrac{x}{2} \right)^2 > x \]
so that $ \R \subset I(f). $  Since all critical points of $ f $ lie on the negative real axis, it follows that none of them are in $ K(f) $.  Thus $ K(f) $ is totally disconnected by Corollary \ref{totdisc}(b).

Now suppose that some component $ U $ of $ F(f) $ is \mconn. Then, for large $ n, $ we have
\[ f^{n+1}(U) \textrm{ surrounds }  f^n(U) \textrm{ which surrounds } 0   \]
by Lemma \ref{baker}, so that $ f^n(U) $ contains no zeros of $ f $ for large $ n $.  However, by Lemma \ref{zheng}, $ f^n(U) $ contains an annulus $ \lbrace z : \alpha_n < \vert z \vert < \beta_n \rbrace $ for large $ n $, where $ \beta_n/\alpha_n \to \infty $ as $ n \to \infty.$   Since the zeros of $ f $ are at $ z = -2^n, n \in \N, $ this is a contradiction and it follows that $ f $ has no \mconn\ Fatou components.
\end{proof}

In Examples \ref{bakerdom} and \ref{infprod} the critical points of $ f $ lie outside $ K(f) $.  This is not essential for $ K(f) $ to be totally disconnected, and in our next example all of the critical points are inside $ K(f) $.

\begin{example}
\label{KSThmB}
Let $ f $ be the \tef\ constructed by Kisaka and Shishikura in \cite[Theorem B]{KS}.  Then $ K(f) $ is totally disconnected.  Moreover, each critical point of $ f $ lies in a strictly preperiodic component of $ K(f)$.
\end{example}

\begin{proof}
In \cite[Theorem B]{KS}, Kisaka and Shishikura used quasiconformal surgery to construct a \tef\ $ f $ with a doubly connected Fatou component which remains doubly connected throughout its orbit.  It follows from Theorem \ref{types}(a) that $ f $ is \spl.

Now the construction of $ f $ in \cite{KS} ensures that all the critical values of $ f $ map to $ 0 $, which is a repelling fixed point.  Furthermore, each critical value of $ f $ lies in the unbounded complementary component of at least one doubly connected Fatou component that surrounds $ 0 $.  Thus the component $ K_0 $ of $ K(f) $ containing $ 0 $ cannot include a critical point, for if it did $ f(K_0) $ would meet a doubly connected Fatou component, which is a contradiction.  Hence each critical point lies in a component of $ K(f) $ which differs from $ K_0 $ and is strictly preperiodic.  It follows from Corollary ~\ref{totdisc}(b) that $ K(f) $ is totally disconnected.
\end{proof}

\begin{remark*}
Recall from Section \ref{intro} that a \tef\ $ f $ is \spl\ whenever the set $ A_R(f) $ is a \sw\ (we again refer to \cite{RS10a} for an explanation of the terminology used here).  In fact, it follows from \cite[Theorem 1.9(a)]{RS10a} that, if $ R>0 $ is such that $ M(r,f)>r $ for $ r \geq R $, then $ A_R(f) $ is a \sw\ for each of the functions in Examples \ref{bakerdom} and \ref{KSThmB}.  Furthermore, it can be shown using \cite[Theorem 1.9(b)]{RS10a} that $ A_R(f) $ is a \sw\ for the function in Example \ref{infprod} (we omit the details). 
\end{remark*}

For our final example in this section, we exhibit a \tef\ which is not \spl, but for which $ K(f) $ is totally disconnected.

\begin{example}
\label{fatou}
Let $ f $ be the function
\[ f(z) = z + 1 + e^{-z}, \]
first investigated by Fatou \cite[Example 1]{F}.  Then $ K(f) $ is totally disconnected.
\end{example}

\begin{proof}
Recall that $ F(f) $ is a completely invariant Baker domain in which $ f^n(z) \to \infty $ as $ n \to \infty $.  As stated in  \cite[Example 3]{RS09}, it can be shown using a result of Bara\'{n}ski \cite[Theorem C]{Bar}, together with the fact that $ f $ is the lift of $ g(w) = (1/e)we^{-w} $ under $ w = e^{-z} $, that:
\begin{itemize}
\item $ J(f) $ consists of uncountably many disjoint simple curves, each with one finite endpoint and the other endpoint at $ \infty $, and
\item $ I(f) \cap J(f) $ consists of the open curves and some of their finite endpoints.
\end{itemize}
Thus all points in $ F(f) $ and all points on the curves to infinity in $ J(f) $, together with some of their finite endpoints, lie in the escaping set $ I(f). $  It follows that $ K(f) $ is a subset of the finite endpoints of the curves to infinity in $ J(f) $.  Thus, if the set of finite endpoints of these curves is totally disconnected, then $ K(f) $ is totally disconnected.

Now it follows from \cite[Theorem 1.5]{BJR} that $ J(f) $ is a Cantor bouquet, in the sense of being ambiently homeomorphic to a subset of $ \R^2 $ known as a straight brush (we refer to \cite{AO, BJR} for a detailed discussion of these ideas).  Now Mayer \cite[Theorem~3]{M} has shown that, if $ h(z) = \lambda e^z, 0 < \lambda < 1/e, $ the set of finite endpoints of $ J(h) $ is totally disconnected.  Since $ J(h) $ is also is a Cantor bouquet, it is ambiently homeomorphic to $ J(f) $.  We conclude that the set of finite endpoints of $ J(f) $ is totally disconnected, and this completes the proof.   
\end{proof}

\section{A non-trivial component of $ K(f) $ with empty interior} 
\label{further}
\setcounter{equation}{0}

In this section, we construct a \tef\ for which $ K(f) $ has a component with no interior that is not a singleton.

We obtain a function with the desired property by modifying a quasiconformal surgery construction of Bergweiler \cite{wB11}, which is itself based on an approach used by Kisaka and Shishikura in \cite{KS} (for which see also Example \ref{KSThmB} above).

The construction uses the following two lemmas on quasiregular mappings -  for background on such mappings we refer to \cite{SR}.  

\begin{lemma} \cite[Theorem 3.1]{KS}, \cite[Lemma 1]{wB11}  
\label{qent}
Let $ g : \C \to \C $ be a quasiregular mapping.  Suppose that there are disjoint measurable sets $ E_j \subset \C,  j \in \N, $ such that:
\begin{enumerate}[(a)]
\item for almost every $ z \in \C $, the $ g $-orbit of $ z $ meets $ E_j $ at most once for every $ j $;
\item $ g $ is $ K_j $-quasiregular on $ E_j; $
\item $ K_\infty := \prod_{j = 1}^\infty K_j < \infty; $
\item $ g $ is analytic almost everywhere outside $ \bigcup_{j = 1}^\infty E_j. $
\end{enumerate}
Then there exists a $ K_\infty $-quasiconformal mapping $ \phi : \C \to \C  $ such that $ f = \phi \circ g \circ \phi^{-1} $ is an entire function.
\end{lemma}

In Lemma \ref{qreg}, $ \log $ denotes the principal branch of the logarithm.
 
\begin{lemma} \cite[Lemma 6.2]{KS}
\label{qreg}
Let $ k \in \N, \, 0 < r_1 < r_2, $ and for $ j = 1,2, $ let $ \phi_j $ be analytic on a \nhd\ of $ \lbrace z : \vert z \vert = r_j \rbrace $ and such that $ \phi_j \vert_{\vert z \vert = r_j} $ goes round the origin $ k $ times. If
\begin{equation}
\label{first}
\left \vert \log \left( \dfrac{\phi_2(r_2 e^{iy})}{r_2^k}\dfrac{r_1^k}{\phi_1(r_1 e^{iy})} \right) \right \vert \leq \delta_0 
\end{equation}
and
\begin{equation}
\label{second}
\left \vert z \dfrac{d}{dz} \left( \log \dfrac{\phi_j(z)}{z^k} \right) \right \vert \leq \delta_1, \quad z = r_j e^{iy}, \quad j = 1,2, 
\end{equation}
hold for every $ y \in (- \pi, \pi] $ and for some positive constants $ \delta_0 $ and $ \delta_1 $ satisfying
\begin{equation}
\label{third}
C = 1 - \dfrac{1}{k} \left( \dfrac{\delta_0}{\log (r_2/r_1)} + \delta_1 \right) > 0, 
\end{equation}
then there exists a quasiregular mapping 
\[ H : \lbrace z: r_1 \leq \vert z \vert \leq r_2 \rbrace \to \C \setminus \lbrace 0 \rbrace \]
without critical points such that $ H = \phi_j $ on $ \lbrace z : \vert z \vert = r_j \rbrace, j = 1,2, $ and satisfying
\[ K_H \leq \dfrac{1}{C}. \]
\end{lemma}

We now give the details of the construction of a \tef\ with the desired property.
 
\begin{example}
\label{qcs}
There exists a \tef\ $ f $ such that $ K(f) $ has a component which has empty interior but which is not a singleton.
\end{example}

\begin{proof}
We first define a quasiregular mapping $ g $ and then obtain the required entire function $ f $ using Lemma \ref{qent}.  

In Bergweiler's construction \cite{wB11}, sequences $ (a_n) $ and $ (R_n) $  are chosen in such a way that $ z \mapsto a_nz^{n+1} $ maps $ \textrm{ann} (R_n, R_{n+1}) $ onto $ \textrm{ann} (R_{n+1}, R_{n+2}) $, where  
\[ \textrm{ann} (r_1, r_2) := \lbrace z \in \C : r_1 < \vert z \vert < r_2 \rbrace, \quad r_2 > r_1 > 0. \]
The mapping $ g $ is then defined by $ g(z) = a_nz^{n+1} $ on a large subannulus of $ \textrm{ann} (R_n, R_{n+1}) $ for each $ n \in \N $, and by interpolation using \cite[Lemma 6.3]{KS} (see also \cite[Lemma 2]{wB11}) in the annuli containing the circles $ \lbrace z : \vert z \vert = R_n \rbrace $ that lie between these subannuli.  We modify Bergweiler's construction only on a disc surrounding the origin.

First we define the boundaries of the various annuli we will need.  Here we follow Bergweiler precisely but we give the details for convenience.  Set $ R_0 = 1. $ Choose $ R_1 > R_0 $ and put
\[ R_{n+1} := \dfrac{R_n^{n+1}}{R_{n-1}^n} \]
for $ n \in \N $.  With $ \gamma = \log R_1 $ we then have
\[ \log \dfrac{R_{n+1}}{R_n} = n \log \dfrac{R_n}{R_{n-1}} = \cdots = n!\log \dfrac{R_1}{R_0} = \gamma n!. \]
Now define sequences $ (P_n), (Q_n), (S_n) $ and $ (T_n) $ by
\begin{equation}
\label{logs}
\log \dfrac{T_n}{S_n} =  \log \dfrac{S_n}{R_n} =  \log \dfrac{R_n}{Q_n} =  \log \dfrac{Q_n}{P_n} = \sqrt{\log \dfrac{R_{n+1}}{R_n}} = \sqrt{\gamma n!}. 
\end{equation}
Setting $ R_1 > e $ gives $ \gamma > 1 $ and so
\[ \dfrac{T_n}{S_n} = \dfrac{S_n}{R_n} = \dfrac{R_n}{Q_n} = \dfrac{Q_n}{P_n} > e.  \]
We also have
\[ \begin{split}
\log \dfrac{P_{n+1}}{T_n} & = - \log \dfrac{Q_{n+1}}{P_{n+1}} - \log \dfrac{R_{n+1}}{Q_{n+1}} + \log \dfrac{R_{n+1}}{R_n} - \log \dfrac{S_n}{R_n} - \log \dfrac{T_n}{S_n} \\
& = -2 \sqrt{\gamma (n+1)!} + \gamma n! - 2 \sqrt{\gamma n!} > 0,\\
\end{split} \]
provided $ R_1 $ and hence $ \gamma $ is sufficiently large.  It follows that
\[ P_n < Q_n < R_n < S_n < T_n < P_{n+1} \]
for all $ n \in \N. $

Now, again following Bergweiler, define sequences $ (a_n) $ and $ (b_n) $ as follows:
\[ a_n := \dfrac{R_{n+1}}{R_n^{n+1}} = \dfrac{1}{R_{n-1}^n}, \]
and
\[ b_n := - \dfrac{(n+1)^2}{n+2} \left( \dfrac{n+1}{n} \right)^n a_n. \]

We will show that there is a quasiregular mapping $ g: \C \to \C $ with the following properties:
\begin{enumerate}[(i)]
\item $ g(z) = z^2 - 2 $ for $ \vert z \vert \leq S_1 $;
\item $ g(z) = a_nz^{n+1} $ for $ T_n \leq \vert z \vert \leq P_{n+1}, n \geq 1 $;
\item $ g(z) = b_n(z-R_n)z^n $ for $ Q_n \leq \vert z \vert \leq S_n, n \geq 2 $;
\item $ g $ is $ K_n $-quasiregular in $ E_n $ for $ n \geq 1 $, where
\[ E_n = \textrm{ann} (S_n, T_n) \cup \textrm{ann} (P_{n+1}, Q_{n+1}) \quad \textrm{and} \quad K_n = 1 + \dfrac{1}{n^2}; \]
\item $ g(\textrm{ann} (S_n, Q_{n+1})) \subset \textrm{ann} (S_{n+1}, Q_{n+2}) $ for $ n \geq 1. $
\end{enumerate} 

Our mapping $ g $ differs from the quasiregular mapping constructed by Bergweiler in \cite{wB11} only in the disc $ \lbrace z : \vert z \vert \leq P_2 \rbrace $.  Bergweiler's mapping was set equal to $ z^2 $ throughout this disc (since $ a_1 = 1 $), whereas our mapping is equal to $ z^2 $ only in the closure of $ \textrm{ann} (T_1, P_2) $ and we have introduced the new function $ z^2 - 2 $ in the smaller disc $ \lbrace z : \vert z \vert \leq S_1 \rbrace. $  Thus Bergweiler's proof that his mapping has the stated properties applies without amendment to our mapping $ g $, but we need to carry out an additional interpolation between the functions $ z^2 - 2 $ and $ z^2 $ in order to define $ g $ in $ \textrm{ann} (S_1, T_1) $.  We also need to check that property (v) still holds for $ n = 1. $

To define $ g $ in $ \textrm{ann} (S_1, T_1) $ we apply Lemma \ref{qreg} with 
\[ \phi_1(z) = z^2 - 2, \quad  \phi_2(z) = z^2, \quad r_1 = S_1 \textrm{ and } r_2 = T_1. \]
Evidently $ k =2 $ in Lemma \ref{qreg}, so (\ref{first}) becomes
\[
\left \vert \log \left( \dfrac{T_1^2 e^{2iy}}{T_1^2}\dfrac{S_1^2}{S_1^2 e^{2iy} - 2} \right) \right \vert
 = \left \vert \log \left( 1 - \dfrac{2}{S_1^2} e^{-2iy} \right) \right \vert.
\] 
Now as $ y $ runs through the interval $ ( -\pi, \pi] $, the point $ z = 1 - \dfrac{2}{S_1^2} e^{-2iy} $ traces out a small circle with centre $ 1 $ (note that $ S_1 > e R_1 > e^2 $).  Thus for such $ z $ we have 
\[ \vert z \vert \leq 1 + \dfrac{2}{S_1^2}\] 
and
\[ \vert \arg z \vert \leq \sin^{-1} \dfrac{2}{S_1^2} \leq \dfrac{\pi}{S_1^2}, \]
so $ \log \vert z \vert < \dfrac{2}{S_1^2}  $ and 
\[ \vert \log z \vert < \sqrt{ \dfrac{4}{S_1^4} + \dfrac{\pi^2}{S_1^4}}  < \dfrac{4}{e^4}. \]
It follows that (\ref{first}) is satisfied with $ \delta_0 = \dfrac{4}{e^4}. $

Moreover for $ j = 1, $ (\ref{second}) becomes
\[ \left \vert z \dfrac{d}{dz} \left( \log \dfrac{z^2 - 2}{z^2} \right) \right \vert = \dfrac{4}{\vert z^2 - 2 \vert } \]
where $ z = S_1e^{iy}. $  But
\[  \dfrac{4}{\vert z^2 - 2 \vert } \leq \dfrac{4}{S_1^2 - 2} < \dfrac{4}{e^4 - 2} \]
so that (\ref{second}) is satisfied with $ \delta_1 = \dfrac{4}{e^4 - 2}. $  For $ j = 2 $, (\ref{second}) is satisfied for any $ \delta_1 > 0 $.

With these values of $ \delta_0 $ and $ \delta_1 $, (\ref{third}) gives
\[ C = 1 - \dfrac{1}{2} \left( \dfrac{4}{e^4 \log (T_1/S_1)} + \dfrac{4}{e^4 - 2} \right) > \dfrac{1}{2}. \]
It follows that there exists a quasiregular mapping
\[ H : \lbrace z: S_1 \leq \vert z \vert \leq T_1 \rbrace \to \C \setminus \lbrace 0 \rbrace \]
without critical points such that $ H(z) = z^2 - 2 $ on $ \lbrace z : \vert z \vert = S_1 \rbrace $ and $ H(z) = z^2 $ on $ \lbrace z : \vert z \vert = T_1 \rbrace $, satisfying
\[ K_H \leq 2. \]
Thus, putting $ g(z) = H(z) $ in $ \textrm{ann} (S_1, T_1) $ we see that (iv) holds for all $ z \in E_1 $, since our definition of $ g $ coincides with Bergweiler's on $ \textrm{ann} (P_2, Q_2) $.  

Next, we check that (v) still holds for $ z \in \textrm{ann} (S_1, Q_2). $  Since our quasiregular mapping $ g $ agrees with Bergweiler's on $ \lbrace z : \vert z \vert = Q_2 \rbrace $, his argument that $ \vert g(z) \vert \leq Q_3 $ for $ z \in \textrm{ann} (S_1, Q_2) $ (which uses the maximum principle) continues to hold.  It therefore remains to show that, for such $ z $, we have $ \vert g(z) \vert \geq S_2. $

Now $ g $ has no zeros in $ \textrm{ann} (S_1, Q_2) $ so if $ z \in \textrm{ann} (S_1, Q_2) $ we have
\[ \vert g(z) \vert \geq \min_{\vert \zeta \vert = S_1} \vert \zeta^2 - 2 \vert \geq S_1^2 - 2, \] 
by the minimum principle.  Moreover, since $ R_1 = e^\gamma $ we have $ S_1 = R_1 e^{\sqrt{\gamma}} = e^{\gamma + \sqrt{\gamma}} $ by (\ref{logs}), and therefore
\[ \vert g(z) \vert \geq e^{2\gamma + 2\sqrt{\gamma}}  - 2 \]
for $ z \in \textrm{ann} (S_1, Q_2) $.
Now 
\[ \log \dfrac{S_2}{R_1} = \log \dfrac{R_2}{R_1} + \log \dfrac{S_2}{R_2}  = \gamma + \sqrt{2 \gamma},  \]
so that $ S_2 = R_1 e^{\gamma + \sqrt{2 \gamma}} = e^{2 \gamma + \sqrt{2 \gamma}}. $  It follows that we can ensure that $ \vert g(z) \vert > S_2 $ for $ z \in \textrm{ann} (S_1, Q_2) $ by choosing $ \gamma $ sufficiently large, and  (v) will then still hold. 

Our mapping $ g $ and the sets $ E_j, j \in \N, $ therefore meet the conditions of Lemma \ref{qent}, and we conclude that there exists a $ K_\infty $-quasiconformal mapping $ \phi : \C \to \C $ such that $ f = \phi \circ g \circ \phi^{-1} $ is an entire function.  Now it follows from (v) that $ g^n(z) \to \infty $ as $ n \to \infty $ for $ z \in \textrm{ann} (S_1, Q_2). $  However, inside the disc $ \lbrace z : \vert z \vert \leq S_1 \rbrace $ the iterates of $ g $ are the iterates of $ z^2 - 2 $.  In particular, the interval $ [-2, 2] $ is invariant under iteration by $ g $ and contains the critical point $ 0 $, whilst for all $ z \in \lbrace z : \vert z \vert \leq S_1 \rbrace \setminus [-2,2] $ there must be some $ N \in \N $ such that $ \vert g^N(z) \vert > S_1 $.

It follows that $ \phi (\textrm{ann} (S_1, Q_2)) $ lies in a \mconn\ component $ U $ of $ F(f) $, whilst $ \phi ([-2,2]) $ is an invariant Jordan arc which is a subset of a component $ K $ of $ K(f) $ containing a critical point.  Now suppose that $ K $ contains some point $ w \notin  \phi ([-2,2]). $  Then there exists $ N \in \N $ such that $ f^N(w) $ lies outside the image under $ \phi $ of the disc $ \lbrace z : \vert z \vert \leq S_1 \rbrace $.  However, as $ f^N(K) $ is connected, this means that $ f^N(K) $ meets $ U $, which is a contradiction since $ U \subset I(f) $ by Lemma \ref{baker}.  Thus $ K $ is a component of $ K(f) $ with empty interior.  This completes the proof.
\end{proof}

\begin{remarks*}
\begin{enumerate}[1.]
\item It follows from \cite[Theorem 1.1(c)]{O10} that every \nhd\ of $ K $ contains a \mconn\ Fatou component that surrounds $ K $, and that $ K $ is a buried component of $ J(f) $.  Since $ f $ is \spl, there are at most countably many components of $ K(f) $ with empty interior that are not singletons by Corollary \ref{totdisc}(a).  
\item Since we have modified Bergweiler's construction only inside the disc $ \lbrace z : \vert z \vert \leq P_2 \rbrace $, the conclusions of \cite{wB11} still hold, and $ f $ has both simply and \mconn\ wandering Fatou components. 
\end{enumerate}
\end{remarks*}

\textbf{Acknowledgements} The author wishes to express his thanks to his doctoral supervisors, Phil Rippon and Gwyneth Stallard, for their assistance in the preparation of this paper, to Alastair Fletcher and Dan Nicks for useful ideas which helped with the proof of Theorem \ref{branhubbgen}, and to Lasse Rempe for suggesting the result in Theorem \ref{uncount}(a).

\end{document}